\documentclass[12pt,reqno,utf8]{amsart}

\usepackage{latexsym}
\usepackage{amsmath}
\usepackage{amssymb, amscd, amsthm, mathtools}
\usepackage[all]{xy}
\usepackage{multirow}
\usepackage{xcolor}
\usepackage{tikz}
\usepackage{tikz-cd}
\usepackage{mathrsfs}
\usepackage{fancyhdr}
\usepackage{cite}
\usepackage{hyperref}
\usepackage[a4paper, hcentering, left=3.5cm, right=3.5cm, top=3.5cm, bottom=3.5cm]{geometry}

\usepackage[shortlabels,inline]{enumitem}
\SetEnumerateShortLabel{a}{\textup{(\alph*)}}
\SetEnumerateShortLabel{A}{\textup{(\Alph*)}}
\SetEnumerateShortLabel{1}{\textup{(\arabic*)}}
\SetEnumerateShortLabel{i}{\textup{(\roman*)}}
\SetEnumerateShortLabel{I}{\textup{(\Roman*)}}

\DeclareMathOperator{\Z}{\mathbb{Z}}
\DeclareMathOperator{\N}{\mathbb{N}}

\DeclareMathOperator{\R}{\mathbb{R}}
\DeclareMathOperator{\bP}{\mathbb{P}}
\DeclareMathOperator{\A}{\mathbb{A}}

\DeclareMathOperator{\C}{\mathbb{C}}

\DeclareMathOperator{\Pics}{\mathbf{Pic}}
\DeclareMathOperator{\Jac}{\mathbf{Jac}}

\DeclareMathOperator{\bO}{\mathcal{O}}

\DeclareMathOperator{\Spec}{\mathrm{Spec}}
\DeclareMathOperator{\Proj}{\mathrm{Proj}}
\DeclareMathOperator{\chara}{\mathrm{char}}

\DeclareMathOperator{\tdeg}{\mathrm{tdeg}}
\DeclareMathOperator{\DEG}{\mathrm{DEG}}
\DeclareMathOperator{\red}{\mathrm{red}}
\DeclareMathOperator{\Div}{\mathrm{Div}}
\DeclareMathOperator{\Frac}{\mathrm{Frac}}
\DeclareMathOperator{\Pic}{\mathrm{Pic}}
\DeclareMathOperator{\sm}{\mathrm{sm}}
\DeclareMathOperator{\sing}{\mathrm{sing}}
\DeclareMathOperator{\zZ}{\mathrm{Z}}

\newtheorem{thm}{Theorem}[section]
\newtheorem{lem}[thm]{Lemma}
\newtheorem{prop}[thm]{Proposition}
\newtheorem{cor}[thm]{Corollary}
\newtheorem{conj}{Conjecture}
\theoremstyle{definition}
\newtheorem{defn}{Definition}[section]
\theoremstyle{remark}
\newtheorem*{rmk}{Remark}
\newtheorem{emp}{Example}[section]

\begin{document}
\title[Finiteness of Pythagoras numbers]{Finiteness of Pythagoras numbers of finitely generated real algebras}
\author{Yi Ouyang$^{1,2}$ , Qimin Song$^1$ and Chenhao Zhang$^1$}
\address{$^1$School of Mathematical Sciences, Wu Wen-Tsun Key Laboratory of Mathematics,   University of Science and Technology of China, Hefei 230026, Anhui, China}

\address{$^2$Hefei National Laboratory, University of Science and Technology of China, Hefei 230088, China}

\email{yiouyang@ustc.edu.cn}	
\email{sqm2020@mail.ustc.edu.cn}
\email{chhzh@mail.ustc.edu.cn}


\thanks{Partially supported by NSFC (Grant No. 12371013) and  Innovation Program for Quantum Science and Technology (Grant No. 2021ZD0302902).}

\begin{abstract} In this paper, we establish two finiteness results and propose a conjecture concerning the Pythagoras number $P(A)$ of a finitely generated real algebra $A$. 
Let $X \hookrightarrow \bP^n$ be an integral projective surface over $\R$, let $\widetilde{X}$ be the normalization of $X$, and let $s \in \Gamma(X,\bO_X(1))$ be a nonzero section such that \( \bigl(\widetilde{X}_{s=0}\bigr)^{\red} \) is formally real.
We prove $P\bigl(\Gamma(X_{s\neq 0})\bigr)=\infty$.
As a corollary, the Pythagoras numbers of integral smooth affine curves over $\R$ are shown to be unbounded.
For any finitely generated $\R$-algebra $A$, if the Zariski closure of the real points of $\Spec(A)$ has dimension less than two, we demonstrate $P(A)<\infty$.	
\end{abstract}
\maketitle

\section{Introduction}
The question of representing positive integers as sums of squares  can be traced back to the work of the Greek mathematician Diophantus in the 3rd century.
Fermat discovered that every prime of the form $4n+1$ is uniquely expressible as the sum of two squares. In $1770$, Lagrange proved his famous four-square theorem that every positive integer can be expressed as the sum of the squares of four integers.

Hilbert's $17\textsuperscript{th}$ Problem asked whether a nonnegative real multivariate polynomial can be represented as a sum of squares of rational functions.
This problem was solved positively by E. Artin.
Moreover, Pfister showed that such polynomials can be represented as sums of $2^n$ squares of rational functions.

Now questions about sums of squares over certain real algebras can be asked in two directions: the representability of elements as sum of squares, and the least number of squares in the sums if they are representable.
 
For the first question, Scheiderer~\cite{Sche01,Sche06} proved that a nonnegative function on a formally real smooth affine surface is a sum of squares when the real part of the surface is compact.
Blekherman, Smith, and Velasco~\cite{Blek16} proved that every nonnegative quadratic form on a formally real integral projective scheme $Y \hookrightarrow \bP^n$ is a sum of squares of linear forms if and only if $Y \hookrightarrow \bP^n$ is a variety of minimal degree.

The second question is about the Pythagoras numbers of real algebras.
Denote the set of elements that can be represented as sums of squares in a commutative ring $A$ by $S(A)$.
Recall that the length $l(a)$ of an element $a \in S(A)$ is the least number of squares representing $a$, and the Pythagoras number $P(A)$ is defined to be $P(A)=\sup\{l(a)\mid a\in S(A)\}$. 

It has been found that there are polynomials in $\R[x,y]$ with arbitrarily large lengths in \cite[Theorem 4.16]{Choi82}, in other words,  $P\bigl(\R[x,y]\bigr)=\infty$.
Furthermore, it was proved in \cite{Choi82} that the Pythagoras number of any regular local ring with formally real residue field of dimension greater than two is infinite.
Examples of real algebras with finite Pythagoras numbers were also given in \cite[Theorem 3.1]{Choi82} and \cite[Theorem 5.20]{Choi82}.
Benoist~\cite{Beno20} showed that $P\bigl(\R((x_1,\dots , x_n))\bigr) \leq 2^{n-1}$, confirming a conjecture raised in \cite{Choi82}.
Hu~\cite{Hu15} proved that $P\bigl(\R((x_1,x_2 , x_3))\bigr)=4$ and the assumption   $P\bigl(\R(x_1,\dots , x_{n-1})\bigr) = 2^{n-1}$ which is Pfister's Conjecture for the case $n-1$ implies the inequality $P\bigl(\R((x_1,\dots , x_n))\bigr) \geq 2^{n-1}$ (and hence the equality $P\bigl(\R((x_1,\dots , x_n))\bigr) = 2^{n-1}$ by Benoist's result). Recently, B\l{}achut and Kowalczyk~\cite{Blac24}  constructed a sequence of polynomials with arbitrarily large lengths and hence proved that the Pythagoras numbers of some algebras like $\R\left[x,y,\sqrt{f(x,y)}\right]$ are infinite. They also proposed a conjecture predicting that the Pythagoras numbers  for more general surfaces are infinite.

In this paper, we first use the geometric method to characterize all possible finitely generated real algebras with finite Pythagoras numbers in the context of \cite[Conjecture 3.3]{Blac24} and Conjecture~\ref{conj1}.
We prove that the Pythagoras number of a finitely generated algebra over $\R$ is finite if the dimension of its formally real part is less than two in Corollary~\ref{cDl1}.

Let $X \hookrightarrow \bP^n$ be an integral projective surface over $\R$, let $\widetilde{X}$ be the normalization of $X$, and let $s \in \Gamma(X,\bO_X(1))$ be a nonzero section such that \( \bigl(\widetilde{X}_{s=0}\bigr)^{\red} \) is formally real.
Applying a real analogue of Bertini's Theorem to construct a sequence of elements whose lengths tend to infinity in Theorem~\ref{MT}, thereby proving $P(X_{s\neq 0})=\infty$. We also give several consequences of this theorem.
First, for every finitely generated formally real field extension $F/\R$ with transcendence degree two, by Corollary~\ref{cnBT} there exists a finitely generated integral $\R$-algebra $A$ satisfying $\Frac(A)=F$ and $P(A)= \infty$. This partially resolves \cite[Problem 4]{Choi82}.
Second, if $X$ is normal, the Pythagoras numbers of integral smooth affine curves contained in $X$ are shown to be unbounded, which completely settles \cite[Problem 3]{Choi82}.
Third, for any polynomial
$f \in \R[x,y,z]$ whose leading homogeneous component $f_0$ is square-free, if all irreducible factors of $f_0$ change sign in $\R^3$, then $P\bigl(\R[x,y,z]/(f)\bigr)=\infty$.

Let $A$ be a finitely generated algebra over $\R$, and let $W$ denote the Zariski closure of the real points of $\Spec(A)$.
Based on observation from Corollary~\ref{cDl1} and Theorem~\ref{MT}, we propose the conjecture that $P(A) <\infty$ if and only if $\dim(W)<2$. 
And we also provide equivalent formulations of this conjecture in Proposition \ref{pCt1}.

\section{Notations}

We assume all commutative rings appearing in this paper are unitary. For a commutative ring $A$, define
	\begin{equation*}
			S_n(A) := \Bigl\{x=\sum_{i=1}^{n}a_i^2 \Big| a_i \in A\Bigr\}, \qquad
			S(A)  := \bigcup_{n\geq 1} S_n(A).
	\end{equation*}
The length $l_A(a)$ of an element $a\in S(A)$ is
	\[ l_A(a)=l(a):=\min\{n \in \N\mid a\in S_n(A)\}. \]	
The Pythagoras number $P(A) \in \N \cup \{\infty\}$ of $A$ is defined as
	\[ P(A):=\sup \{l_A(a)\mid a \in S(A)\}.\]

For a scheme $X$, the Pythagoras number is $P(X)=P\bigl(\Gamma(X,\bO_X)\bigr)$.

\begin{defn}	
A commutative ring $A$ is called formally real if
	\[ S(A) \cap \bigl(-S(A)\bigr) = \{0\}. \]
A point $p$ in a scheme is called formally real if the residue field of $p$ is formally real.	
\end{defn}

Note that a field $k$ is formally real if and only if $-1$ cannot be represented as a sum of squares, which implies that $\chara(k)=0$.
A formally real field is equivalent to an ordered field (see\cite[Chapter VIII, Proposition 1.3]{Lam05}).

\begin{defn}
Let $X$ be a Noetherian scheme over a field $k$.
\begin{enumerate}[i]
 \item $X$ is called formally real if $X=X^{\red}\neq\emptyset$ and the generic points of $X$ are formally real, and called nonreal if otherwise.
	
 \item	$X$ is called purely nonreal if no generic point of $X$ is formally real.
\end{enumerate}
\end{defn}

Let $X$ be an algebraic scheme over $\C$.
Then there is a natural complex topology on $X(\C)$ that gives a complex analytic structure on $X(\C)$ (see \cite[APPENDIX B.1]{Hart77} and \cite[\S 2.5]{Serr56}).
If $X$ is a quasi-projective scheme, then there is an immersion $X \to \bP^n$.
The complex topology on $X(\C)$ is then simply the subspace topology induced by the standard complex topology on $\bP^n(\C)$.

\begin{defn}
Let $X$ be a formally real algebraic scheme over $\R$.
Let $\mathcal{L}$ be an invertible sheaf on $X$ and $s \in \Gamma (X,\mathcal{L})$.

We say that $s$ does not change sign in $X$ if, for an affine cover $X_{\sm}=\bigcup_{i=1}^nU_i$, where $\mathcal{L}$ has a trivialization on each $U_i$, the function $s|_{U_i}$ does not change sign in each connected component of $U_i(\R)$ under the standard topology.
Otherwise, we say that $s$ changes sign in $X$.
\end{defn}
Let $X$ be a scheme.
Let $\mathcal{L}$ be an invertible sheaf on $X$ and $s \in \Gamma (X,\mathcal{L})$.
We denote the zero locus of $s$ in $X$ by $X_{s=0}$.
If $s$ is a regular section, we also use $\zZ(s)$ instead of $X_{s=0}$ when emphasizing the effective Cartier divisor defined by $s$.
We denote the open subscheme $X\setminus X_{s=0}$ by $X_{s\neq 0}$.

\begin{defn}\label{def2_1}
	Let $A$ be a commutative ring, and let $C \subseteq A$ be a subring.
	A real degree function $d$ on $A$ with constants in $C$ is a map $A \to \Z \cup \{-\infty\}$ satisfying:
	\begin{enumerate}[i]
		\item $d(f)= -\infty \Leftrightarrow f=0$;
		\item $d(C \setminus \{0\})=\{0\}$;
		\item $d(fg)=d(f)+d(g)$ for $f,g \in A$;
		\item $d(f^2+g^2)=2\cdot \max\{d(f),d(g)\}$ for $f,g \in A$.
	\end{enumerate}
If $A$ is a $k$-algebra where $k$ is a formally real field, a real degree function on $A$ is, by default, defined with constants in $k$.	
\end{defn}
\begin{rmk}
If $A$ is a commutative ring with a real degree function $d$, then by (i), (iii) and (iv) of Definition~\ref{def2_1}, $A$ must be a formally real integral domain.
Moreover, for $f, g \in A$, we have
	\[ \begin{split}
		2\cdot \max\{d(f+g),d(f-g)\}=&d\left(\frac{(f+g)^2+(f-g)^2}{2}\right)\\ =&d(f^2+g^2)=2\cdot \max\{d(f),d(g)\}.
	\end{split} \]
	Hence, $d(f+g) \leq \max\{d(f),d(g)\}$. 
    In particular, if $d(f) \neq d(g)$, then $d(f+g) = \max\{d(f),d(g)\}$.
\end{rmk}
\begin{emp} \label{eRDG}
	\begin{enumerate}[a]
		\item For the ring $\R[x_1, \dots , x_m]$, the total degree function $\tdeg$ and the degree functions $\deg_{x_i}$ (for $1\leq i\leq m$) with respect to $x_i$ are all real degree functions.
		\item If the ring  $B=\R[x_1, \dots, x_m]/I$, where $I$ is a homogeneous prime ideal, is formally real, then the total degree function $\tdeg$ on $B$ induced by the grading on $\R[x_1, \dots , x_m]$ is a real degree function.
		\item If $A$ is a subring of $B$ and $d$ is a real degree function on $B$ with constants in $C$,
		then $d|_A$ is a real degree function on $A$ with constants in $A\cap C$.
		\item If $A$ is a formally real integral domain with a real degree function $d$ with constants in $C$, then its field of fractions $\Frac(A)$ has an induced real degree function $d_{\Frac(A)}$ with constants in $\Frac(C)$ given by 
		 $d_{\Frac(A)}\bigl(\frac{f}{g}\bigr)=d(f)-d(g)$.
	
        \item Let $A$ be a finitely generated formally real integral algebra of positive dimension over $\R$, and let $X$ be an integral projective model of $\Spec(A)$.
        Denote by $X_{\mathrm{frs}}^{(1)}$ the set of height one formally real smooth points in $X$.
        These points correspond to discrete valuation rings of $\Frac(A)$ with formally real residue fields.
        Let $\DEG_{\R}(A)$ denote the set of real degree functions on $A$.
        There exists a natural injective map \[X_{\mathrm{frs}}^{(1)} \hookrightarrow \DEG_{\R}\bigl(\Frac(A)\bigr) \to \DEG_{\R}(A)\] defined by $q \mapsto -v_q$, where $v_q$ is the valuation associated to $q$.
        Note that the set $X_{\mathrm{frs}}^{(1)}$ is uncountable by Corollary~\ref{cnBT}.
	\end{enumerate}
\end{emp}

\section{Formally real points on algebraic schemes}
\begin{prop}\label{prop1_2}
Let $k$ be a field with $\chara(k) \neq 2$, and let $A$ be a finitely generated reduced $k$-algebra. The following are equivalent:
	\begin{enumerate}[1]
		\item $A$ is formally real;
        \item the generic points of $\Spec(A)$ are formally real;
		\item the set of formally real points in $\Spec(A)$ is Zariski dense;
		\item the set of formally real closed points in $\Spec(A)$ is Zariski dense.
	\end{enumerate}
\end{prop}

\begin{proof} $(1) \Leftrightarrow (2)$.
	Let $M$ be a multiplicative set of $A$.
    Let $M^{-1}S_n(A)$ denote the subset $\{\frac{s}{m^2} \mid s \in S_n(A),\ m \in M\}$ of $S(M^{-1}A)$.
    Let $M^{-1}S(A)= \cup_{n\geq 1} M^{-1}S_n(A)$.
	Then $M^{-1}S(A)=S(M^{-1}A)$, which implies that localizations of formally real algebras are formally real. Hence, we have $(1) \Rightarrow (2)$.
    Let $\eta_i$ be the generic points of $\Spec(A)$, where $i=1, \dots , s$.
	Since $A$ is reduced, the map $A \to \prod_{i=1}^{s}k_{\eta_i}$ is injective. This proves $(2) \Rightarrow (1)$.

\medskip \noindent $(2) \Rightarrow (3)$ is trivial.
	
\medskip \noindent $(3) \Rightarrow (4)$ follows by the Artin-Lang homomorphism. Let $V$ be an affine open subset of $\Spec(A)$. By condition (3), there exists a formally real point $p$ in $V$. Then, $k$ is a subfield of $k_p$ and thus formally real. Let $R'$ be a real closure of $k_p$. By \cite[Chapter VIII, Theorem 2.17]{Lam05}, there is a real closure $R$ of $k$ contained in $R'$. Therefore, there exists an $R$-rational point in $V$ by \cite[Theorem 4.1.2]{Boch98}.
	
\medskip \noindent $(4) \Rightarrow (2)$. Let $B$ be a regular local ring with a formally real residue field $k$. Let $\hat{B}$ be the completion of $B$ at its maximal ideal.
Then $\hat{B}=k[[x_1,\dots ,x_n]]$ by Cohen's structure theorem, hence is formally real. Via the inclusion $B \hookrightarrow \hat{B}$, $B$ and its field of fractions are formally real.

The smooth locus $U$ of $\Spec(A)$ is  dense open subset in $\Spec(A)$. Let $\coprod_{i=1}^{s} U_i$ be the decomposition of connected components of $U$, and let $\eta_i$ be the generic point of $U_i$. Assuming (4), there exists a formally real point $p_i$ in every $U_i$; consequently $\eta_i$ is a formally real point.
\end{proof}

\begin{cor}\label{cRP}
    Let $S=\bigoplus_{d \geq 0}S_d$ be a finitely generated positive-dimensional reduced graded algebra over a field $S_0=k$.
    Then $S$ is formally real if and only if $X=\Proj(S)$ is formally real.
\end{cor}
\begin{proof}
    Denote the prime ideal $\bigoplus_{d > 0}S_d$ by $S_+$.
    Since $X$ is quasi-compact, it is covered by finitely many standard open affine subschemes $D_{+}(f_i)$, where $i=1,\dots, n$ and each $f_i$ is a homogeneous element in $S_+$.
    By \cite[\href{https://stacks.math.columbia.edu/tag/00JP}{Tag 00JP}]{stacks-project}, $\sqrt{(f_1,\dots, f_n)}=S_+$.
    Denote $\Spec(S)\setminus\{S_+\}$ by $\A(X)$.
    Then we obtain an affine cone $\phi :\A(X) \to X$ defined locally by $S_{(f_i)} \to S_{f_i}$.
    
    Since the irreducible components of $X$ correspond to minimal homogeneous prime ideals of $S$, each generic point of $\A(X)$ maps to a generic point of $X$.  
    The fiber of $\phi$ at a point $p \in X$ is $\A^{1}_{k_p}\setminus\{0\}$.
    Thus, the generic points of $\Spec(S)$ are formally real if and only if the generic points of $X$ are formally real.
\end{proof}

\begin{lem} \label{lZtL}
    Let $X$ be a separated algebraic scheme over $\R$. Let $\dim(X)$ be the Krull dimension of $X$ and $\dim\bigl(X(\R)\bigr)$ be the Lebesgue covering dimension of $X(\R)$. 
    Then $\dim(X) \geq \dim\bigl(X(\R)\bigr)$.
    
    Moreover, if $X$ is irreducible, then $\dim(X)=\dim\bigl(X(\R)\bigr)$ if and only if $X^{\red}$ is formally real.
\end{lem}
\begin{proof}
    Since $X$ is separated, $X(\C)$ is Hausdorff.
    Moreover, $X(\C)$ is second countable and locally metrizable.
    Therefore, by the Nagata-Smirnov metrization theorem, $X(\C)$ is metrizable.
    As a subspace of $X(\C)$, $X(\R)$ is metrizable.
    
    We proceed by induction on the Krull dimension of $X$.
    If $\dim(X)=-1$, then $X=\emptyset$ by definition.
    Therefore, $\dim\bigl(X(\R)\bigr)=-1$.
    
    Denote the smooth locus of $X^{\red}$ by $X^{\red}_{\sm}$.
    If there is a $p \in X^{\red}_{\sm}(\R)$ then, by implicit function theorem, there exists a open subset $U_p \ni p$  in $X^{\red}_{\sm}(\R)$, such that $\R^{\dim_p(X)}\xrightarrow{\sim} U_p$.
    Note that $\dim_p(X)\leq \dim(X)$ is the Krull dimension of $X$ at $p$ (see \cite[\href{https://stacks.math.columbia.edu/tag/0055}{Tag 0055}]{stacks-project}).
    By the fundamental theorem of dimension theory, $\dim(U_p)=\dim_p(X)$ (see \cite[1.8.2]{Enge78}).
    Hence, by \cite[1.1.1]{Enge78} and \cite[1.7.7]{Enge78}, $\dim\bigl(X^{\red}_{\sm}(\R)\bigr) \leq \dim(X)$.
    Since $X(\R)$ is a metric space, $X^{\red}_{\sm}(\R)$ is a $F_{\sigma}$ subset. 
    Thus, \[\dim\bigl(X(\R)\bigr)\leq \max\bigl\{\dim\bigl(X^{\red}_{\sm}(\R)\bigr),\dim\bigl(X_{\sing}(\R)\bigr) \bigr\} \leq \dim(X)\] by \cite[1.5.4]{Enge78}, \cite[1.7.7]{Enge78} and induction.
    
    If $X^{\red}$ is irreducible and formally real, then $X^{\red}_{\sm}(\R) \neq \emptyset$.
    For every $p \in X(\C)$, $\dim_p(X) = \dim(X)$.
    Thus, $\dim(X)=\dim\bigl(X(\R)\bigr)$ by \cite[1.1.2]{Enge78}.
\end{proof}

\begin{lem}\label{lSaR}
Let $X$ be a formally real integral separated algebraic scheme over $\R$. For $\mathcal{L} \in \Pic(X)$ and $0 \neq s \in \Gamma(X, \mathcal{L})$, $s$ does not change sign in $X$ if $X_{s=0}$ is purely nonreal.

Moreover, if $X$ is regular in codimension one and $X_{s=0}$ is reduced, then the converse is true.
\end{lem}
\begin{proof}
If $s$ changes sign in $X$, then there exists an affine open subscheme $U \subseteq X_{\sm}$ and a trivialization $\mathcal{O}_U \xrightarrow{\sim} \mathcal{L}|_U$ such that $s$ changes sign in a connected open subset $V\subseteq U(\R)$.
The map  $s: U(\R) \to \R$ is continuous with respect to the standard topology.
Let $V_{>}=\{x\in V \mid s(x)<0\}$, $V_{<}=\{x\in V \mid s(x)>0\}$ and $V_{=}=X_{s=0}(\R) \cap V$.

Since the dimension of $X_{s=0}(\R)$ is less than $\dim(V)$ by Lemma~\ref{lZtL}, $X_{s=0}(\R) \cap V$ is a nowhere dense subset of $V$.
Since $V$ is connected, the intersection of boundaries, $V_0=\partial_{V}(V_{>}) \cap \partial_{V}(V_{<})$, is nonempty.

Let $p\in V_0$.
Then $s$ changes sign in every open neighbourhood of $p$.
By implicit function theorem, we may assume $V$ is homeomorphic to $\R^{\dim(X)}$.
Thus, $\dim(V_{=})=\dim(X)-1$ by \cite[1.8.12]{Enge78}.
    It follows that $\dim\bigl(X_{s=0}(\R)\bigr)\geq \dim(X)-1$ by \cite[1.1.2]{Enge78}.
    Note that $\dim(X_{s=0})=\dim(X)-1$.
    Thus, $X_{s=0}(\R)$ is Zariski dense in some irreducible components of $X_{s=0}$ by Lemma~\ref{lZtL}.
    Therefore, $X_{s=0}$ is not purely nonreal by Proposition~\ref{prop1_2}.
    
Now assume that $X$ is regular in codimension one and $X_{s=0}$ is reduced and not purely nonreal .
Note that $X$ is regular in codimension one means that the codimension of $X_{\sing}$ in $X$ is greater than one.

Let $\eta$ be a formally real generic point of $X_{s=0}$.
Note that $\eta \in X_{\sm}$.
Let $W \subseteq X_{\sm}$ be an open affine subscheme containing $\eta$ such that $W \cap X_{s=0}$ is integral smooth and $\mathcal{L}$ is trivial on $W_0$.
There exists $q \in X_{s=0} \cap W(\R)$ because $X_{s=0}$ is formally real.
Then, by implicit function theorem, there exists an open subset $E \in W(\R)$ containing $q$ with a diffeomorphism $E \xrightarrow{\sim} \R^{\dim(X)}$.
Due to the smoothness of $X_{s=0}$, the gradient of $s$ as a function on $E \xrightarrow{\sim} \R^{\dim(X)}$ cannot vanish at any point of $X_{s=0} \cap E$.
Thus, $s$ changes sign on $E$ by inverse function theorem (or by the differential approximation equation).
\end{proof}

Let $V$ be a vector space over $\R$, and let $P$ be a property of vectors in $V$.
We say $P$ holds for general $v \in V$ if there exists a nonempty Zariski open subset $U \subseteq V$ such that $P$ holds for $v \in U$.
We say $P$ holds for sufficiently many $v \in V$ if there exists a nonempty Euclidean open subset $U \subseteq V$ such that $P$ holds for $v \in U$.

Now we establish a real analogue of Bertini's Theorem via Lemma~\ref{lSaR}.
\begin{prop}\label{pNSB}
Let $X$ be a formally real integral separated algebraic scheme over $\R$. 
Let $\mathcal{L}$ be an invertible sheaf.
If $\dim\bigl(\Gamma(X, \mathcal{L})\bigr) \geq 2$, then for sufficiently many sections $s \in \Gamma(X, \mathcal{L})$ the scheme $X_{s=0}$ is not purely nonreal.
\end{prop}
\begin{proof}
We need to find an element $s'$ in $\Gamma(X, \mathcal{L})$ such that $X_{s'=0}$ is not purely nonreal and prove the existence of a Euclidean open subset $W$ containing $s'$.

Let $U \subseteq X_{\sm}$ be an affine open subscheme, and let $\{s_i\}_{i=1,\dots,n}$ be an $\R$-basis of $\Gamma(X, \mathcal{L})$.
Let $V$ be a connected component of $U(\R)$.
Then every nonzero linear combination of $s_1$ and $s_2$ is not zero on $V$, because $s_1$ and $s_2$ are linear independent and $V$ is Zariski dense in $X$.
Thus, we have $p \in V$ such that $(s_1(p), s_2(p))\neq 0$, and we have $q\in V$ such that $s_1(p)s_2(q)-s_2(p)s_1(q) \neq 0$.
Then the matrix $\begin{pmatrix}
                   s_1(p) & s_2(p) \\
                   s_1(q) & s_2(q)
                 \end{pmatrix}$ is invertible.
Thus, there exists $k_1, k_2 \in \R$, such that $(k_1s_1+k_2s_2)(p)<0$ and $(k_1s_1+k_2s_2)(q)>0$.
    
The functions   $e_p(x)=\sum_{i=1}^{n}s_i(p)x_i$ and $e_q(x)=\sum_{i=1}^{n}s_i(q)x_i$ are linear and hence continuous on $\R^n$. For $v=(k_1,k_2,0,\dots,0)$, one has $e_p(v) < 0$ and $e_q(v) > 0$. Thus, there exists an open subset $W \in \Gamma(X, \mathcal{L})$ containing $v$ such that $e_p(W)\subseteq (-\infty, 0)$ and $e_q(W)\subseteq (0, +\infty)$.
    By Lemma \ref{lSaR}, for each $s \in W$, $X_{s=0}$ is not purely nonreal.
\end{proof}

\begin{cor}\label{cnBT}
    Let $X \hookrightarrow \bP^n_{\R}$ be a formally real integral quasi-projective scheme of dimension greater than one over $\R$.
    Then for sufficiently many sections $H \in \Gamma(\bP^n_{\R}, \bO_{\bP^n}(d))$ the scheme $X \cap H$ is formally real and integral.
    Consequently, for sufficiently many sections $s \in \Gamma(X, \bO_X(d))$ the scheme $X_{s=0}$ is formally real and integral.
\end{cor}
\begin{proof}
    $X$ is geometrically irreducible because the function field of $X$ is formally real.
    A hypersurface that has a smooth proper intersection with $X_{\sm}$ is general in $\Gamma(\bP^n_{\R}, \bO_{\bP^n}(d))$ (see \cite[\href{https://stacks.math.columbia.edu/tag/0FD6}{Tag 0FD6}]{stacks-project}).
    A hypersurface that has a geometrically irreducible intersection with $X$ is general in $\Gamma(\bP^n_{\R}, \bO_{\bP^n}(d))$ (see \cite[\href{https://stacks.math.columbia.edu/tag/0G4F}{Tag 0G4F}]{stacks-project}).
    It's not difficult to see that the dimension of the image of $\Gamma(\bP^n_{\R}, \bO_{\bP^n}(d)) \to \Gamma(X, \bO_X(d))$ is greater than two (or see the first paragraph of the proof of \cite[\href{https://stacks.math.columbia.edu/tag/0G4F}{Tag 0G4F}]{stacks-project}).
    A Zariski open subset of a vector space is dense under the standard topology.
    Thus, by Proposition~\ref{pNSB}, for sufficiently many sections $H \in \Gamma(\bP^n_{\R}, \bO_{\bP^n}(d))$ the scheme $X \cap H$ is formally real and integral.
    
    The proof of $\Gamma(X, \bO_X(d))$ follows the same arguement as above.
\end{proof}

Let $C$ be a one-dimensional integral algebraic scheme over a field $k$.
Let $U$ be an open subscheme of $C_{\sm}$.
Let $\Div(U) \subseteq \Div(C)$ be the set of divisors supported on $U$.
Then there is a natural map $\Div(U) \to \Pic(C)$ by sending a divisor $D$ to the dual $\bO_C(D)$ of its fractional ideal, which is an invertible subsheaf of $\underline{K(C)}$.
Denote $\dim_kH^i(C,\bO_C(D))$ by $h^i_C(D)$, where $D$ is a divisor in $\Div(U)$.
Subscripts of $h^i_C$ may be suppressed if no ambiguity arises.

\begin{lem} \label{lBoG}
    Let $X \hookrightarrow \bP^n_{\R}$ be a projective surface with the Hilbert Polynomial $H_{\bO_X}(x)=\chi(X,\bO_X(x))=\frac{k}{2}x^2+\alpha x+ \beta$.
    Let $C$ be an integral curve in $X$ defined by a section $s \in \Gamma(X,\bO_X(d))$ and $g_a=g_a(C)$ be the arithmetic genus of $C$.
    
    Then $g_a=\frac{k}{2}d^2-\alpha d+1$.
\end{lem}
\begin{proof}
    By Serre Vanishing Theorem, there exists an integer $t$ such that for all $m \geq t$, $H^{1}(X,\bO_X(m-d))=0$ and $H^{1}(C,\bO_C(m))=0$ (see \cite[III, Proposition 5.3]{Hart77}).
    Let $\mathcal{I}$ be the defining ideal sheaf of $C$.
    We have $\mathcal{I} = \bO_X(-d)$.
    Therefore, for every integer $m \geq t$, there is an exact sequence
    \[0 \to \Gamma(X,\bO_X(m-d))\to\Gamma(X,\bO_X(m))\to \Gamma(C,\bO_C(m)) \to 0.\]
    Then $\chi(C,\bO_C(m))=H(m)-H(m-d)$ when $m \geq t$.
    
    Thus, $\deg\bigl(\bO_C(m)\bigr)=\chi(C,\bO_C(m))-\chi(C,\bO_C)=-\frac{k}{2}d^2+mkd+\alpha d - (1-g_a)$ (see \cite[\href{https://stacks.math.columbia.edu/tag/0AYR}{Tag 0AYR}]{stacks-project}).
    By B\'{e}zout's Theorem and \cite[\href{https://stacks.math.columbia.edu/tag/0AYY}{Tag 0AYY}]{stacks-project}, $\deg\bigl(\bO_C(1)\bigr)=\deg(C)=kd$.
    Hence, $\deg\bigl(\bO_C(m)\bigr)=m\deg\bigl(\bO_C(1)\bigr)=mkd$.
    We conclude that $g_a=\frac{k}{2}d^2-\alpha d+1$.
\end{proof}

We formulate a singular version of a result in Milne~\cite[VII, Lemma 5.2]{Corn86}.
\begin{lem}\label{lSD}
    Let $C$ be an integral curve over a field $k$.
    \begin{enumerate}[1]
      \item If $D \in \Div(C_{\sm})$ such that $h^1(D) > 0$, then there exists a nonempty open subscheme $U$ of $C_{\sm}$ such that $h^1(D+Q)=h^1(D)-1$ for all $k$-rational points $Q$ in $U$.
      \item For every $r \leq g_a(C)$, there exists a nonempty open subscheme $U$ of $C^r_{\sm}$ such that $h^0(\sum_{i=1}^{r}P_i)=1$ for all $k$-rational points $(P_1,\dots,P_r)$ in $U$.
    \end{enumerate}
\end{lem}\begin{proof}
    (1) Since $C$ is integral of dimension one, $C$ is Cohen-Macaulay.
    Let $\omega_C$ be a dualizing sheaf on $C$ (see \cite[III, Proposition 7.5]{Hart77} or \cite[\href{https://stacks.math.columbia.edu/tag/0C10}{Tag 0C10}]{stacks-project}).
    By duality theorem, we have $h^1(E)=\dim_k\Gamma(C,\bO(-E) \otimes \omega_C)$ for every $E \in \Div(C_{\sm})$.
    
    Since $h^1(D) > 0$, $\Gamma(C,\bO(-D)\otimes \omega_C) \neq 0$.
    Let $0 \neq s \in \Gamma(C,\bO(-D)\otimes \omega_C)$, and let $U=C_{\sm}\setminus C_{s=0}$.
    Let $Q$ be a closed point of $U$.
    Then by the definition, there is an exact sequence \[ 0 \to \bO(-Q) \to \bO_C \to i_{Q,*}k_Q \to 0.\]
    Note that $\omega_C |_{C_{\sm}}=\omega_{C_{\sm}}$  is an invertible sheaf on $C_{\sm}$ (see the proof of \cite[III, Theorem 7.11]{Hart77} or, \cite[\href{https://stacks.math.columbia.edu/tag/0AU0}{Tag 0AU0}]{stacks-project} and \cite[\href{https://stacks.math.columbia.edu/tag/0FVV}{Tag 0FVV}]{stacks-project}).
    Taking the tensor product of $\bO(-D)\otimes \omega_C$ and the above exact sequence, we get \[ 0 \to \bO(-Q-D)\otimes \omega_C \to \bO(-D)\otimes \omega_C \to i_{Q,*}k_Q  \to 0.\]
    Thus we have an exact sequence of global sections \[ 0 \to \Gamma(\bO(-Q-D)\otimes \omega_C) \to \Gamma(\bO(-D)\otimes \omega_C) \to k_Q.\]
    The morphism $\Gamma(\bO(-D)\otimes \omega_C) \to k_Q$ is nontrivial because $s \neq 0$ at $Q$.
    When $Q$ is a $k$-rational point, the above morphism is surjective; hence, $h^1(D+Q)=h^1(D)-1$.
    
    (2) We may assume $r >1$.
    We need only find an open subscheme $U$ of $C_{\sm}^r$ such that $h^1(\sum_{i=1}^{r}P_i)=g_a-r$ for all $k$-rational points $(P_1,\dots,P_r)$ in $U$ by \cite[\href{https://stacks.math.columbia.edu/tag/0AYR}{Tag 0AYR}]{stacks-project} and \cite[\href{https://stacks.math.columbia.edu/tag/0AYY}{Tag 0AYY}]{stacks-project}.
    
    Denote by $V$ the vector space $\Gamma(C,\omega_C)$.
    Let $V_{P}$ denote the linear system $\Gamma\bigl(\bO\bigl(-\sum_{i=1}^{t}P_i\bigr)\otimes \omega_C\bigr)$, where $P=(P_1,\dots, P_t)$ is a $k$-rational point of $C_{\sm}^t$.
    Denote by $P_{(t')}$ the $k$-rational point $(P_1,\dots, P_{t'})$ of $C_{\sm}^{t'}$, where $P=(P_1,\dots, P_t)$ is a $k$-rational point of $C_{\sm}^t$ and $t' \leq t$.
    Following the argument in (1), we need only find a nonempty open subscheme $U$ of $C_{\sm}^r$ such that for each $P=(P_1,\dots,P_r) \in U(k)$ and each nonnegative integer $t<r$, $P_{t+1}$ is not a base point of $V_{P_{(t)}}$.
    
    Let $C_b$ be an nonempty open subscheme of $C_{\sm}$ such that $V$ generates $\omega_{C_b}$.
    Then $V$ induces a morphism $\phi: C_b \to \bP^{g_a-1}_k$ (see \cite[II, Theorem 7.1]{Hart77}).
    Let $H$ be the closed subscheme of $\bigl(\bP^{g_a-1}\bigr)^r$ defined by the $r\times r$ minors of the $r\times g_a$ matrix of indeterminates \[\begin{pmatrix}
    X_{1,0} & \cdots & X_{1,g_a-1} \\
    \vdots & \ddots & \vdots \\
    X_{r,0} & \cdots & X_{r,g_a-1} 
    \end{pmatrix}\]
    (rigorously, this is defined by the Segre embedding $\bigl(\bP^{g_a-1}\bigr)^r \hookrightarrow \bP^{g_a^r-1}$).
    Since $r \leq g_a$, $H$ is a proper subscheme of $\bigl(\bP^{g_a-1}\bigr)^r$.
    
    Denote by $U_0$ the open subscheme $\bigl(\bP^{g_a-1}\bigr)^r \setminus H$.
    As $\phi$ is induced by $V$ and $C$ is integral, $\phi(C_b)$ is not contained in any hyperplane of $\bP^{g_a-1}_k$.
    Thus, $U_0 \cap \phi^r(C_b^r)$ is nonempty.
    Therefore, $U=C_b^r \setminus (\phi^r)^{-1}(H)$ is the desired nonempty open subscheme.
\end{proof}

\begin{prop}\label{pSoR}
    Let $X \hookrightarrow \bP^n_{\R}$ be a projective surface with the Hilbert function $H_{\bO_X}(x)=\frac{k}{2}x^2+\alpha x+ \beta$.
    Let $C$ be a formally real integral curve defined by a section in $\Gamma(X,\bO_X(d))$, where $d\in \Z^+$.
    There exists an integer $N_X$ (depending only on $X$) such that for every integer $m \geq d+N_X$ there is a section $s \in \Gamma(X,\bO_X(m))$ whose restriction $s |_C$ has support contained in $C_{\sm}(\R)$.
    
    If $X$ is normal, the results holds for all integers \[m \in \left[\frac{d}{2}+N_X,d-N_X\right] \cup \left[d+N_X, +\infty\right).\]
\end{prop}
\begin{proof}
    By Serre Vanishing Theorem, there exists an integer $n_X$ such that for all $m \geq n_X$, $H^{1}(X,\bO_X(m))=0$.
    Let $\mathcal{I}$ be the defining ideal sheaf of $C$.
    We have $\mathcal{I} = \bO_X(-d)$.
    Therefore, for every integer $m \geq d+n_X$, the map $\Gamma(X,\bO_X(m)) \to \Gamma(C,\bO_C(m))$ is surjective.
    
    Since $C$ is a proper flat and geometrically integral scheme over $\R$, and $C$ has dense real points, the relative picard functor of $C/\R$ is represented by a locally quasi-projective scheme $\Pics_{C/\R}$ (see \cite[9.2]{Fant05} and \cite[Corollary 9.4.18.3]{Fant05}).
    Let $\Jac_{C/\R}$ be the connected component of $\Pics_{C/\R}$ containing the identity element, then $\Jac_{C/\R}$ is a smooth quasi-projective algebraic group of dimension $g_a(C)$ (see \cite[Theorem 9.5.11]{Fant05}).
    
    Let $C_{\sm}^{(b)}$ denote the $b^{th}$ symmetric power of $C_{\sm}$ (see \cite[VII, \S3]{Corn86} or \cite[Remark 9.3.9]{Fant05}).
    If $E \in \Div(C_{\sm})$ is a divisor of degree $b >0$, then we can define a morphism $\phi_E : C_{\sm}^{(b)} \to \Jac_{C/\R}$ functorially by \[(P_1,\dots, P_b) \mapsto [P_1+\cdots+P_b - E]\] (see \cite[VII, \S5]{Corn86} or \cite[Definition 9.4.6]{Fant05}).
    Note that the morphism $\phi_E$ is proper (see \cite[Exercise 9.4.12]{Fant05} ).
    
    Denote the degree of $C$ by $\delta$.
    By B\'{e}zout's theorem, $\delta= \deg\bigl(\bO_C(1)\bigr)=kd$.
    We can choose a hyperplane $H \in \Gamma(\bP^n_{\R},\bO_{\bP^n}(1))$ such that $H \cap (C\setminus C_{\sm})=\emptyset$.
    Hence, there exists an effective divisor $D \in \Div(C_{\sm})$ representing $\mathcal{O}_C(1)$.
    Since $C$ is formally real, there exists a point $P \in C_{\sm}(\R)$.
    Let $D_m= mD-(m \delta -g_a)P$.
    Then the map $\phi_{D_m} : C_{\sm}^{(g_a)} \to \Jac_{C/\R}$ is injective on an open subscheme of $C_{\sm}^{(g_a)}$ by Lemma~\ref{lSD}(2) (see a similar argument in smooth case in \cite[VII, Lemma 5.2]{Corn86}).
    The schemes $C_{\sm}^{(g_a)}$ and $\Jac_{C/\R}$ share the same dimension.
    Thus, $\phi_{D_m}$ is a birational proper morphism for every $m\geq 0$.
    
    Since $\Jac_{C/\R}$ is smooth, every nonempty Euclidean open subset of $\Jac_{C/\R}(\R)$ has the covering dimension $g_a$ by implicit function theorem.
    Hence, every proper Zariski closed subset of $\Jac_{C/\R}(\R)$ is nowhere dense, as it has a lower covering dimension by Lemma~\ref{lZtL}.
    Since $\phi_{D_m}$ is a birational morphism for every $m \geq 0$, it is an isomorphism over a dense open subscheme of $\Jac_{C/\R}$.
    Therefore, $\phi_{D_m}\bigl(C_{\sm}^{(g_a)}(\R)\bigr)$ is dense in $\Jac_{C/\R}(\R)$.
   
    As $\phi_{D_m}$ is a proper morphism of algebraic schemes, it maps closed subsets of $C_{\sm}^{(g_a)}(\C)$ to closed subsets of $\Jac_{C/\R}(\C)$ (see \cite[APPENDIX B.1]{Hart77} or \cite{Serr56}).
    Consequently, $\phi_{D_m}$ surjects $C_{\sm}^{(g_a)}(\R)$ onto $\Jac_{C/\R}(\R)$ for every $m$, which includes the identity element of $\Jac_{C/\R}$.
    Since $C$ has a real point, the map $\Pic(C) \to \Pics_{C/\R}(\R)$ is an isomorphism (see \cite[9.2]{Fant05}).
    Thus, for every positive integer $m \geq \frac{g_a}{\delta}$, there exists an effective divisor supported on $C_{\sm}(\R)$ linearly equivalent to $mD$.
    By the definition of the linear equivalence of the Cartier divisors, there exists a section in $\Gamma(C,\bO_C(m))$ supported on $C_{\sm}(\R)$ when $m \geq \frac{g_a}{\delta}$.

    By Lemma \ref{lBoG}, the inequality $m \geq \frac{d}{2}+1-\min\{0,\alpha\}$ guarantees the existence of a section in $\Gamma(C,\bO_C(m))$ supported on $C_{\sm}(\R)$.
    Define $N_X:=\left\lceil\max \{n_X, 1-\alpha\}\right\rceil$.
    For $m\geq d+N_X$, we conclude that there exists a section $s \in \Gamma(X,\bO_X(m))$ such that $s |_C$ is supported on $C_{\sm}(\R)$.
    
    If $X$ is normal, by Enriques-Severi-Zariski Lemma, there exists an integer $n_X'$ such that for all $m \geq n_X'$, $H^{1}(X,\bO_X(-m))=0$.
    Consequently, the restriction map $\Gamma(X,\bO_X(m)) \to \Gamma(C,\bO_C(m))$ is surjective for every integer $m \leq d-n_X'$.
    Redefine $N_X:=\left\lceil\max \{n_X,n_X', 1-\alpha\}\right\rceil$.
    We conclude that for all integers \[m \in \left[\frac{d}{2}+N_X,d-N_X\right] \cup \left[d+N_X, +\infty\right)\] there exists a section $s \in \Gamma(X,\bO_X(m))$ with $s |_C$ supported on $C_{\sm}(\R)$.
\end{proof}

\section{Main results}

\begin{thm}\label{tRtR}
Let $k$ be a field with $\chara(k)\neq 2$, and let $A$ be a finitely generated algebra over $k$. Let $W \subseteq \Spec(A)$ be the Zariski closure of formally real points of $\Spec(A)$. Then there exist $x_1, \dots , x_n \in A$, such that $\sum_{i=1}^{n}x_i^2=0$ and $W= \Spec\bigl(A/(x_1,\dots , x_n)\bigr)^{\red}$.
\end{thm}
\begin{proof} We use Noetherian induction to prove the theorem.
	
Let $\Spec(A)$ be denoted by $X$, and $\Spec\bigl(A/(t_1,\dots , t_m)\bigr)$ by $V(t_1,\cdots, t_m)$.  Assume the proposition holds for all proper subschemes of $X$.
We may assume $W \subsetneqq X^{\red}$; otherwise, there is nothing to prove.
	
First, we prove that there exist $y_1,\dots ,y_m \in A$ such that $\sum_{i=1}^{m}y_i^2=0$ and $V(y_1,\dots , y_m)^{\red}\subsetneqq X^{\red}$.
By Proposition~\ref{prop1_2}(1) and (3) and our assumption $W \subsetneqq X^{\red}$, there exist $\bar{a}_1,\dots , \bar{a}_s \in A^{\red} \setminus \{0\} $ such that $\sum_{i=1}^{s}\bar{a}_i^2=0$ in $A^{\red}$ .
Lift $\bar{a}_i$ to $a_i$ in $A$.
Then there exists an integer $N$ such that $\Bigl(\sum_{i=1}^{s}a_i^2\Bigr)^N=0$.

Expanding the expression, we obtain $\sum_{i=1}^{m}y_i^2=0$ in $A$.
Note that the set $\{y_i \mid i=1, \dots, m\}$ contains $\{a_i^N \mid i=1, \dots, s\}$.
Since $a_i$ is nonzero in $A^{\red}$ for each $i$, $V(a_i)^{\red}$ defines a proper subscheme of $X^{\red}$.
Thus, \[V(y_1,\dots , y_m)^{\red} \subseteq V(a_1,\dots , a_s)^{\red} \subsetneqq X^{\red}.\]
Since $\sum_{i=1}^{m}y_i^2=0$ must be trivial at every formally real point, each $y_i$ vanishes on all formally real points. Hence, $W \subseteq V(y_1,\dots , y_m)$.

Let $I=(y_1^2,\dots , y_m^2)$, and let $B=A/I$.
By Noetherian induction, there exist $b_1, \dots , b_s \in A$ such that \[\sum_{i=1}^{s}b_i^2 \in I\quad \text{and} \quad \bigl(V(b_1,\dots , b_s) \cap \Spec(B)\bigr)^{\red}=W.\] 
Then, there exist $c_1, \dots , c_m$ such that $\sum_{i=1}^{s} b_i^2+\sum_{i=1}^{m}c_iy_i^2=0$.
Using the identity
 \[ c_iy_i^2=\left(\frac{c_i+1}{2}\right)^2y_i^2- \left(\frac{c_i-1}{2}\right)^2y_i^2 =\left(\frac{c_i+1}{2}y_i\right)^2+\sum_{j \neq i}\left(\frac{c_i-1}{2}y_j\right)^2,\]
we see that $\sum_{i=1}^{s}b_i^2+\sum_{i=1}^{m}c_iy_i^2$ is a sum of squares $\sum_{i=1}^{n}z_i^2$. Given $\sum_{i=1}^{m}y_i^2 + \sum_{i=1}^{n}z_i^2=0$, we have	
 \[ W \subseteq V(y_1,\dots , y_m,z_1, \dots , z_n)^{\red} \subseteq \bigl(V(b_1,\dots , b_s) \cap \Spec(B)\bigr)^{\red} = W. \]
Therefore, the theorem is proved.
\end{proof}

\begin{lem} \label{lStE}
    Let $A$ be a commutative ring such that $2$ is a unit in $A$.
    Let $a_1, \dots , a_n \in A$.
    Then \[P\bigl(A/(a_1^2+ \dots + a_n^2)\bigr) < \infty \iff P\bigl(A/(a_1^2, \dots , a_n^2)\bigr) < \infty.\]
\end{lem}
\begin{proof}
    See \cite[Corrollary 2.6]{Choi82}.
\end{proof}

\begin{cor}\label{cRtR}
Let $k$ be a formally real field with $P(k) < \infty$, and let $A$ be a finitely generated algebra over $k$.
Let $x_1,\cdots , x_n \in A$ be the elements chosen in Theorem~\ref{tRtR}. 
Then the Pythagoras number $P(A)$ is finite if and only if $P\bigl(A/(x_1^2,\dots , x_n^2)\bigr)$ is finite.
\end{cor}
\begin{proof}
    We have $\sum_{i=1}^{n}x_i^2=0$ in Theorem~\ref{tRtR}.
    Therefore, by Lemma~\ref{lStE}, $P(A)=P\bigl(A/(x_1^2,\dots , x_n^2)\bigr)$.
\end{proof}

\begin{prop}\label{plD}
Let $k$ be a formally real field with $P(k) < \infty$, and let $A$ be a finitely generated algebra over $k$.
\begin{itemize}
  \item If $\dim(A) = 0$, then $P(A) < \infty$.
  \item If $\dim(A) \leq 1$ and $k$ is real closed, then $P(A) < \infty$.
\end{itemize}
\end{prop}
\begin{proof}
    See \cite[Theorem 2.7]{Choi82}.
\end{proof}
We can extend the result of \cite[Theorem 2.7]{Choi82} using Theorem~\ref{tRtR}.
\begin{cor}\label{cDl0}
Let $k$ be a formally real field with $P(k) < \infty$, and let $A$ be a finitely generated algebra over $k$.
If $\Spec(A)$ has only finitely many formally real points, then $P(A)$ is finite.
\end{cor}
\begin{proof}
Let $x_1,\dots , x_n \in A$ be the elements chosen in Proposition~\ref{cRtR}.
The Krull dimension of $A/(x_1^2,\dots , x_n^2)$ is zero by assumption.
Thus $P\bigl(A/(x_1^2,\dots , x_n^2)\bigr)$ is finite by Proposition~\ref{plD}.
\end{proof}

\begin{cor}\label{cDl1}
Let $k$ be a real closed field, and let $A$ be a finitely generated algebra over $k$.
Let $W$ be the Zariski closure of the $k$-points of $\Spec(A)$.
If $\dim(W) \leq 1$, then $P(A)$ is finite.
\end{cor}
\begin{proof}
The proof follows the same approach as Corollary~\ref{cDl0} utilizing Proposition~\ref{plD} again.
\end{proof}

\begin{rmk}
We can always obtain the dimension of the real points of an affine algebraic scheme over $\R$ by the algorithm \cite{Lair21}.
\end{rmk}

\begin{emp} The following examples satisfy assumptions in Corollary~\ref{cDl1}.
    \begin{enumerate}[a]
		\item If $f$ is a square-free semidefinite polynomial in $\R[x,y,z]$, then we have $P\bigl(\R[x,y,z] / (f) \bigr)  < \infty$.
        As a special case, $P\bigl(\R[x,y,z]/(z^2+f(x,y))\bigr) < \infty$ when $f(x,y)$ is nonnegative. This resolves Problem 3.2(b) in \cite{Blac24}.
		\item Let $I$ be the ideal $(x^2y^4+x^4y^2-3z^2+3w^2-6xyw+1, z+xy-w)$ in $\R[x,y,z,w]$.
        Then the Motzkin polynomial $x^2y^4+x^4y^2-3x^2y^2+1$, which is nonnegative on $\R^2$, is contained in the ideal. 
        Hence, the dimension of the real points on the surface $\Spec(\R[x,y,z,w]/I)$ is less than two.
        Thus, $P\bigl(\R[x,y,z,w]/I\bigr) < \infty$.
	\end{enumerate}
\end{emp}

\begin{lem}\label{lEtR}
    Let $A$ be a formally real normal domain and $X=\Spec(A)$.
    Let $U=\Spec(B)$ be an affine open subset of $X$ such that $X^{(1)}\setminus U$ consists of formally real points, where $X^{(1)}$ is the set of height one points in $X$.
    Then $P(B) \geq P(A)$.
    
    Moreover, if $B$ is a localization of $A$, then $P(B) = P(A)$.
\end{lem}
\begin{proof}
    Let $a \in S(A)$ such that $l_A(a)=m$.
    Let $a= \sum_{i=1}^{n} b_i$ be an expression in $B$.
    Let $q \in X^{(1)}\setminus U$.
    Then $v_q(b_i)\geq 0$, because $q$ is a formally real smooth point.
    Therefore, there exists an open subscheme $V$ of $X$ containing $X^{(1)}$ such that $b_i$ is defined on $V$ for all $i$.
    Then $b_i$ is defined on $X$ because $\Gamma(V,\mathcal{O}_V)=A$ (see \cite[\href{https://stacks.math.columbia.edu/tag/031S}{Tag 031S}]{stacks-project}, \cite[\href{https://stacks.math.columbia.edu/tag/0AVZ}{Tag 0AVZ}]{stacks-project} and the distinguished triangle in \cite[\href{https://stacks.math.columbia.edu/tag/0A39}{Tag 0A39}]{stacks-project}).
    Hence, $n \geq m$.
    Thus, $P(B) \geq P(A)$.
    
    If $B=M^{-1}A$, where $M$ is a multiplicative subset of $A$, then $S_n(B)=M^{-1}S_n(A)$.
    Therefore, $P(B) \leq P(A)$ (see the first paragraph of the proof of Proposition~\ref{prop1_2}).
\end{proof}

\begin{lem}\label{lGtA}
    Let $X=\Spec(A)$ be an affine scheme, and let $(f) \subseteq A$ be a radical ideal.
    Let $g \in A$.
    If $V(f) \subseteq V(g)$ set-theoretically, then $f \mid g$ in $A$.
\end{lem}
\begin{proof}
    Since $V(f)$ is reduced, there is a canonical morphism $V(f) \hookrightarrow V(g)^{\red}$.
    The morphism of schemes $V(f) \hookrightarrow V(g)^{\red} \hookrightarrow V(g)$ corresponds to the ring homomorphism $A/(g) \twoheadrightarrow A/(f)$.
    Thus, $g \in (f)$.
\end{proof}

Now we prove the main theorem of our paper.
\begin{thm}\label{MT}
Let $X$ be an integral projective surface over $\R$, $\bO_X(1)$ a very ample invertible sheaf on $X$, and $\widetilde{X}$ the normalization of $X$. Let $U=X_{s\neq 0}$ be the open subscheme of $X$ defined by a nonzero section $s \in \Gamma(X,\bO_X(1))$.
If $\bigl(\widetilde{X}_{s=0}\bigr)^{\red}$ is formally real, then $P(U)=\infty$.
\end{thm}
\begin{proof}
Since the generic points of $\widetilde{X}_{s=0}$ are formally real smooth height one points of $\widetilde{X}$, $\widetilde{X}$ is formally real.

By Corollary~\ref{cnBT}, there exists $s' \in \Gamma(X,\bO_X(1))$ such that $X_{s'=0}$ is formally real and integral, and the generic point $\omega$ of $X_{s'=0}$ lies in $X_{\sm}$.
Since $\bO_X(1)$ is very ample, $U$ is affine.
We have $\Gamma(U\cap X_{s\neq 0})=\Gamma(U)_{\frac{s'}{s}}$.
Thus, $P(U\cap X_{s'\neq 0}) \leq P(U)$.
We need only show that $P(U\cap X_{s'\neq 0}) = \infty$.
Let $V$ be the normalization of $X_{s'\neq 0}$, and let $V_0$ be the normalization of $U\cap X_{s'\neq 0}$.
Then the generic points of $V \setminus V_0$ are formally real by assumption.
Denote $\Gamma(V)$ by $B$.
By Lemma~\ref{lEtR}, for $a \in S(B)$, $l_B(a)=l_{\Gamma(V_0)}(a)$.

Denote $\Gamma(X_{s'\neq 0})$ by $A$.
If we can find $f_m \in S(A)$ such that $l_B(f_m) \geq m$, then $l_{\Gamma(V_0)}(f_m) \geq m$ and $l_{\Gamma(U\cap X_{s'\neq 0})}(f_m) \geq m$.

We proceed by induction on the lengths in $B$ of elements in $A$.
For the base case, $l_B(1)=1$.
Assume we have found $0 \neq g \in S(A)$ with $l_B(g)=m-1$, where $m>1$.

Let $\deg_{\omega}=-v_{\omega}$ be the real degree function as in Example~\ref{eRDG}(e).
Choose an even integer $d>\max\{2N_X,deg_{\omega}(g)\}$, where $N_X$ is defined in Proposition~\ref{pSoR}.
By Corollary~\ref{cnBT}, there exists $f' \in \Gamma(X,\bO_X(d))$ such that  $X_{f'=0}$ is formally real and integral, and the generic point of $X_{f'=0}$ lies in $X_{\sm}$.
Then by Proposition~\ref{pSoR}, there exists $t' \in \Gamma\bigl(X,\bO_X\bigl(\frac{3}{2}d\bigr)\bigr)$ such that $X_{f'=0} \cap X_{t'=0}$ consists of real smooth points of $X_{f'=0}$.
Define $f=\frac{f'}{s'^d}$ and $t=\frac{t'}{s'^{\frac{3}{2}d}}$, noting that $f,t \in A$.

Let $G=t^2+f^2g \in S(A) \cap S_m(B)$.
Assume that $G = \sum_{i=1}^{n} G_i^2$ in $B$, where $n$ is a positive integer.
To prove $l_B(G)=m$, we need to show $n \geq m$.

Define $K= \frac{G}{t^2}$ and $K_i=\frac{G_i}{t}$.
Note that the generic point of $\widetilde{X}\setminus V$ is $\omega$.
Since the generic points of $\widetilde{X}_{t' \neq 0} \setminus V$ are formally real, by the definition of $d$, $K$ is well-defined on an open subscheme of $\widetilde{X}_{t' \neq 0}$ containing $(\widetilde{X}_{t' \neq 0})^{(1)}$.
By Lemma~\ref{lEtR}, $K \in \Gamma(\widetilde{X}_{t' \neq 0})$ and $K_i\in \Gamma(\widetilde{X}_{t' \neq 0})$.
Let $C$ be the affine curve $\widetilde{X}_{f'=0} \cap \widetilde{X}_{t' \neq 0}$.
Since $\widetilde{X}_{f'=0}$ is formally real, $C(\R)$ is nonempty.
Let $p \in C(\R)$. 
We have $1 = \sum_{i=1}^{m-1} K_i(p)^2$.
After an orthogonal transformation, assume $K_1(p)=1$ and $K_i(p)=0$ for $i>1$ . Define $a_1= K_1 -1$ and  $a_i= K_i$ for $i>1$.
Let $b_1= t a_1=G_1-t$ and $b_i= t a_i=G_i$ for $i>1$.

Let $q \in \widetilde{X}_{f'=0} \setminus C$.
By the definition of $t'$, $q$ is a real smooth point in $\widetilde{X}_{f'=0}$.
Following the argument in Lemma~\ref{lEtR}, for each index $i$, the function $a_i$ extends to a global section over the entire space $\widetilde{X}_{f'=0}$.
Due to the isomorphism $\R\xrightarrow{\sim} \Gamma(\widetilde{X}_{f'=0})$, we have $a_i=0$.
    Let $\gamma$ be the generic point of $\widetilde{X}_{f'=0}$. 
    we conclude that $b_i(\gamma)=0$.
    By Lemma~\ref{lGtA}, $f \mid b_i$ in $B$.
	
    Let $b_i = f c_i$ in $B$.
	We have $G = t^2+f^2g=t^2+2ftc_1+ \sum_{i=1}^{n}f^2c_i^2$.
	Rearranging, $t c_1=\frac{1}{2} \cdot f \cdot (g -\sum_{i=1}^{n}b_i^2)$.
    Since $t(\gamma) \neq 0$, $c_1(\gamma)=0$.
    By Lemma~\ref{lGtA} again, $f \mid c_1$ in $B$.
	Let $c_1=fh$.
	Substituting, we get $g=2th+f^2h^2+ \sum_{i=2}^{n}c_i^2$.
    We have $\deg_{\omega}(b) \geq 0$ for all $b \in B \setminus \{0\}$ by the definition of $B$ and the isomorphism $\R\xrightarrow{\sim}\Gamma(\widetilde{X},\bO_{\widetilde{X}})$.
    Note that $\deg_{\omega}(t)\leq \frac{3}{2}\deg_{\omega}(f)<\deg_{\omega}(f^2)$.
	If $h \neq 0$, then \[\deg_{\omega}(f^2h^2) > \deg_{\omega}(2th) \quad \text{and} \quad \deg_{\omega}\Bigl(f^2h^2+ \sum_{i=2}^{n}c_i^2\Bigr) \geq \deg_{\omega}(f^2h^2).\]
	Therefore, $\deg_{\omega}(g) \geq \deg_{\omega}(f^2h^2) \geq \deg_{\omega}(f^2)$ by the remark of Definition \ref{def2_1}.
	However $\deg_{\omega}(f^2)=2\deg_{\omega}(f)>\deg_{\omega}(g)$, forcing $h=0$.
    Since $l_B(g)=m-1$, $n \geq m$.
\end{proof}

\begin{rmk}
    It should be noted that the conditions of the theorem are not necessary.
    Let $X$ be the normal projective surface $\Proj\bigl(\R[x,y,z,w]/(w^2z^2-x^4-y^4-w^4)\bigr)$, and let $U=X_{w \neq 0}$.
    Then $P(U)=\infty$ (see \cite[Theorem 2.3]{Blac24}).
    Here, $X_{w = 0}$ is a reduced purely nonreal projective curve.
\end{rmk}

\begin{cor}\label{cMSC}
    Let $X \hookrightarrow \bP^n_{\R}$ be a formally real integral normal projective surface, and let $s \in \Gamma(X,\bO_X(d))$ be a nonzero section (where $d$ is a positive integer), such that $(X_{s=0})^{\red}$ is formally real.
    Then $P(X_{s \neq 0})= \infty$.
\end{cor}

Now we provide a criterion for Theorem~\ref{MT} without computing the normalizations.

\begin{lem}\label{lItN}
    Let $X$ be an integral projective scheme over $\R$, and let $U$ be an open affine subscheme of $X$.
    Let $H=X\setminus U$ be the reduced closed subscheme of $X$.
    Denote the normalization of $X$ and $U$ by $\widetilde{X}$ and $\widetilde{U}$, respectively.
    If the boundary $\partial_{X(\C)}\bigl(U(\R)\bigr) \subseteq H(\R)$, taken in the standard topology, is Zariski dense in $H$, then the same statement holds for $\widetilde{U} \hookrightarrow \widetilde{X}$.
    Moreover, $X$ and $H$ are formally real.
\end{lem}
\begin{proof}
    Let $H_{>}$ denote the part of $H$ with codimension greater than one in $X$, and let $H_1$ be the equidimensional part of $H$ of codimension one in $X$.
    We begin by proving that $H_{>}=\emptyset$.
    Let $\phi : X \to \widetilde{X}$ be the normalization morphism.
    Define $V=\phi^{-1}(X\setminus H_1)$.
    Then $\widetilde{U}$ is an open subscheme of $V$ containing $V^{(1)}$.
    The proof of Lemma~\ref{lEtR} yields $\Gamma(V)=\Gamma(\widetilde{U})$, which induces a morphism $V \to \Spec\bigl(\Gamma(V)\bigr)=\widetilde{U}$.
    The morphism $V \to \widetilde{U}$ naturally admits a section $\widetilde{U} \hookrightarrow V$, implying that $\widetilde{U} \hookrightarrow V$ is a closed immersion.
    Consequently, we obtain $\widetilde{U}=V$, and therefore $H_{>}=\emptyset$.
    
    Since $\overline{U_{\sing}(\R)} \cap H$ lies in a proper closed subscheme of $H$ with dimension $\dim(H)-1$, it follows that $\partial_{X(\C)}\bigl(U_{\sm}(\R)\bigr)$ is Zariski dense in $H$.
    Define $M=\phi^{-1}(H)$.
    By Chevalley's Theorem \cite[\href{https://stacks.math.columbia.edu/tag/00FE}{Tag 00FE}]{stacks-project}, $\phi\colon M \to H$ is an open morphism.
    By \cite[Chapter~3,\S3]{Grau84}, $\phi\colon M(\C) \to H(\C)$ is an open holomorphic map.
    
    Consider a point $p \in M(\C)$ with $\phi(p) \in \partial_{X(\C)}\bigl(U_{\sm}(\R)\bigr)$, and let $V \subseteq M(\C)$ be a Euclidean open subset containing $p$.
    Then $\phi(V)$ is a Euclidean open subset in $H(\C)$.
    Thus, $U_{\sm}(\R) \cap \phi(V) \neq \emptyset$.
    Hence, $p \in  \partial_{\widetilde{X}(\C)}\bigl(\widetilde{U}_{\sm}(\R)\bigr)$.
    Note that $\phi$ is an open morphism in the Zariski topology, which implies that the inverse image of a dense subset is dense.
    We conclude that $\partial_{\widetilde{X}(\C)}\bigl(\widetilde{U}_{\sm}(\R)\bigr)$ is Zariski dense in $M$.
    
    Since $X$ is projective, $H \neq \emptyset$.
    If $X$ is nonreal, then the real points of $U$ are contained in a proper subscheme $C$ of $U$.
    Then, the Krull dimension of $\overline{C(\R)} \cap H$ is less than $\dim(H)$, which is not Zariski dense in $H$.
    Since the real points of $H$ are Zariski dense in $H$, $H$ is formally real.
\end{proof}

\begin{cor}\label{cMT}
    Let $X \hookrightarrow \bP^n$ be an integral projective surface over $\R$, and let $s \in \Gamma(X,\bO_X(1))$ be a nonzero section such that the boundary $\partial_{X(\C)}\bigl(X_{s\neq 0}(\R)\bigr)$, taken in the standard topology, is Zariski dense in $X_{s=0}$.
    
    Then $P\bigl(\Gamma(X_{s\neq 0})\bigr)=\infty$.
\end{cor}
\begin{proof}
    Let $\widetilde{X}$ be the normalization of $X$.
    By Lemma~\ref{lItN}, $\bigl(\widetilde{X}_{s=0}\bigr)^{\red}$ is formally real.
\end{proof}

Let $X$ be an algebraic scheme over a field $k$, and let $\mathcal{L}$ be an invertible sheaf on $X$.
Define $\Gamma_*(X,\mathcal{L})=\bigoplus_{n \geq 0} \Gamma(X, \mathcal{L}^n)$.
We define \[P_{\mathcal{L}}(m)=\sup\left\{l_{\Gamma_*(X,\mathcal{L})}(a) | a \in \Gamma(X, \mathcal{L}^m)\right\}.\]
Note that if $X$ is a projective scheme over $\R$ and $\mathcal{L}$ is very ample, then $P_{\mathcal{L}}(m)$ is finite (see \cite{Choi95}).
\begin{cor} \label{cHMT}
    Let $X$ be a formally real integral projective surface over $\R$, and let $\mathcal{L}$ be an ample invertible sheaf on $X$.
    Then \[\lim_{m \to +\infty}P_{\mathcal{L}}(2m)=\infty.\]
\end{cor}
\begin{proof}
    Denote $\Gamma_*(X,\mathcal{L})$ by $B$.
    There exists $N \in \N$, such that for each $a \geq N$ $\mathcal{L}^a$ is very ample.
    Let $s \in \Gamma(X,\mathcal{L}^N)$ such that $X_{s=0}$ is formally real and integral, and the generic point of $X_{s=0}$ lies in $X_{\sm}$.
    Since $X_{s=0}$ and $\widetilde{X}_{s=0}$ are birational, $\widetilde{X}_{s=0}$ is formally real.
    Therefore, $P(X_{s \neq 0}) = \infty$.
    
    Let $n \in \N$.
    By the proof of Theorem~\ref{MT}, we obtain $g \in \Gamma(X_{s \neq 0})$ such that $l_{\Gamma(X_{s \neq 0})}(g)>n$.
    
    By \cite[\href{https://stacks.math.columbia.edu/tag/01Q5}{Tag 01Q5}]{stacks-project} and \cite[\href{https://stacks.math.columbia.edu/tag/01W6}{Tag 01W6}]{stacks-project}, $X=\Proj(B)$.
    Note that $B$ is an integral domain and $\Gamma(X_{s \neq 0})=B_{\left((s)\right)}$.
    Hence, there exists an integer $m_0$ such that $g'=gs^{2m_0} \in \Gamma(X,\mathcal{L}^{2m_0N})$ and $g'=\sum_{i=1}^{t}g_i^2 \in S(B)$.
    By Corollary~\ref{cRP}, $B$ being formally real implies that each $g_i$ is a homogeneous element of $B$.
    Defining $f_i=\frac{g_i}{s^{m_0}}$, we obtain $g=\sum_{i=1}^{t}f_i^2$.
    Consequently, $t > n$.
    
    Let $N_0 \geq N$ be an integer.
    Then $\mathcal{L}^{N_0}$ is very ample.
    There exists $s' \in \Gamma(X,\mathcal{L}^{N_0})$ such that $X_{s'=0}$ is formally real and integral, and the generic point of $X_{s'=0}$ lies in $X_{\sm}$.
    Since $B/(s')$ is a formally real integral domain, for any $f \in B$ the equality $l_B(f)=l_B\bigl(f(s')^2\bigr)$ holds.
    We conclude that $P_{\mathcal{L}}(2m+2N_0)\geq P_{\mathcal{L}}(2m)$ for all $m\in \Z^+$.
\end{proof}

\begin{lem}\label{lDoC}
    Let $X \hookrightarrow \bP^n_k$ be an integral projective scheme over a field $k$.
    Let $s_1 \in \Gamma(X,\bO_X(m_1))$ and $s_3 \in \Gamma(X,\bO_X(m_3))$ be global sections satisfying the following conditions:
    \begin{enumerate}[1]
      \item $X$ is either a curve with $\zZ(s_1) \subseteq X_{\sm}$ or a normal scheme;
      \item the divisor $\zZ(s_3)-\zZ(s_1)$ is effective on $X_{\sm}$.
    \end{enumerate}
    Then there exists $s_2 \in \Gamma(X,\bO_X(m_3-m_1))$ satisfying $s_1s_2=s_3$.
\end{lem}
\begin{proof}
    Denote by $B_d=\Gamma(X,\bO_X(d))$ and $B=\Gamma_*(X,\bO_X(1))$.
    Fix a nonzero section $x_0 \in \Gamma(X,\bO_X(1))$. 
    For every nonnegative integer $d$, this induces a natural injective homomorphism $B_d \hookrightarrow K(X), x \mapsto \frac{x}{x_0^d}$, which extends to a ring homomorphism $\phi:B \to K(X)$.
    Geometrically, this corresponds to choosing a $K(X)$-rational point in the generic fiber of the affine cone $\A(X) \to X$.
    
    Let $ x_1, \dots, x_t$ be nonzero global sections in $\Gamma(X,\bO_X(1))$ such that $X=\bigcup^{t}_{i=0}X_{x_i \neq 0}$.
    Denote by $U_i$ the open subset $X_{x_i\neq 0}$.
    For any integer $r \in \Z$, the twisted sheaf $\bO_X(r)$ can be viewed as an invertible subsheaf $\mathcal{L}_r$ of $\underline{K(X)}$ with local trivializations $\mathcal{L}_r|_{U_i}=\phi(x_i^r)\bO_{U_i}$.
    We note that $\Gamma(X,\mathcal{L}_d)=\phi(B_d)$ holds for all nonnegative integers $d$. 
    
    Define $m_2=m_3-m_1$.
    When $X$ is a curve satisfying $\zZ(s_1) \subseteq X_{\sm}$, since the divisor $\zZ(s_3)-\zZ(s_1)$ is effective on $X_{\sm}$, the rational function $\frac{\phi(s_3)}{\phi(s_1x_i^{m_2})}=\frac{s_3}{s_1x_i^{m_2}}$ is regular on each $U_i$.
    When $X$ is normal, for each $i$, the rational function $\frac{s_3}{s_1x_i^{m_2}}$ is regular on an open subscheme $V_i \subseteq U_i$ containing all height one points, and thus it is regular on $U_i$ following the argument in Lemma~\ref{lEtR}.
    These results imply that $\frac{\phi(s_3)}{\phi(s_1)}$ defines a global section of $\mathcal{L}_{m_2}$.
    Consequently, $s_2:=\frac{s_3}{s_1}$ lies in the graded component $B_{m_2}$.
\end{proof}

\begin{prop}
    Let $X \hookrightarrow \bP^n_{\R}$ be a formally real integral normal projective surface over $\R$.
    The Pythagoras numbers of formally real integral smooth affine curves contained in $X$ are unbounded.
\end{prop}
\begin{proof}
    Let $M$ be a positive integer.
    By Lemma~\ref{cHMT}, there exists an integer $n_0$ such that $P_{\bO_X(1)}(2m) > M$ for all integers $m \geq n_0$.
    
    By Enriques-Severi-Zariski Lemma (or duality theorem), there exists a positive integer $n_X$ such that $H^1(X,\bO_X(-m))=0$ and $\Gamma(X,\bO_X(-m))=0$ for all $m \geq n_X$.
    Let $d>2N_X+4n_0+2n_X$ be an even integer, and let $d'=\frac{d}{2}+N_X$, where $N_X$ is defined in Proposition~\ref{pSoR}.
    Since $X_{\sing}$ is a finite set, following the argument in Corollary~\ref{cnBT}, there exists a section $s \in \Gamma(X,\bO_X(d))$ such that the scheme $X_{s=0}$ is formally real, integral and smooth.
    Denote $X_{s=0}$ by $C$.
    By Proposition~\ref{pSoR}, there exists a section $t \in \Gamma(X,\bO_X(d'))$ with $C_{t=0} \subseteq C(\R)$.
    
    Let $d_1 \in \left[n_0, \frac{d-n_X}{2}\right]$ be an integer, and let $f \in \Gamma(X,\bO_X(2d_1))$ satisfy $l_{\Gamma_*(X,\bO_X(1))}(f)>M$.
    Define $B_m=\Gamma(C,\bO_C(m))$ and $B=\Gamma_*(C,\bO_C(1))$.
    By the definition of $d_1$, $\Gamma(X,\bO_X(d_1)) \to B_{d_1}$ and $\Gamma(X,\bO_X(2d_1)) \to B_{2d_1}$ are isomorphisms.
    Therefore, $l_B(f)=l_{\Gamma_*(X,\bO_X(1))}(f) > M$.
    
    Suppose $t^{2m}f=\sum_{i=1}^{M_0} (f_i')^2$ in $B$, where $m \in \Z^+$.
    Since $B$ is formally real, each $f_i'$ is homogeneous in $B$.
    As $C_{t=0} \subseteq C(\R)$, the divisor $\zZ(f_i')-\zZ(t^m)$ is effective for every $i$.
    By Lemma~\ref{lDoC}, there exist $f_i \in B_{d_1}$ such that $f_i'=f_it^m$.
    This implies $l_B(f)=l_B(t^{2m}f)>M$ for all $m \in \N$.
    Consequently, $\varliminf_{m \to +\infty}P_{\bO_C(1)}(2m)>M$.
    
    Therefore, there exists an integer $d_2 > d'$ and $g \in B_{2d_2}$ satisfying $l_B(g)>M$.
    By Proposition~\ref{pSoR} again, there exists a section $t_1 \in B_{d_2}$ with $C_{t_1=0} \subseteq C(\R)$.
    Following the argument above, $l_B(g)=l_B(t_1^{2m}g)>M$ for all $m \in \N$.
    This implies $l_{C_{t_1 \neq 0}}(\frac{g}{t_1^2})=l_B(g)>M$.
\end{proof}

\begin{prop}
    Let $f$ be a polynomial of total degree $n>0$ in $\R[x,y,z]$.
    Let $f_i$ be the homogeneous part of degree $(n-i)$ of $f$.
    Let $f_0=\prod_{k=1}^{m}D_k^{j_k}$ be an irreducible decomposition of $f_0$, where each $D_k$ is an irreducible homogeneous polynomial and $j_k> 0$.
    Define $v_k=\min\bigl\{i,2n+1 \big| D_k \nmid f_i\bigr\}$.
    Assume that each $D_k$ changes sign in $\R^3$.
    If one of the following conditions holds:
    \begin{enumerate}[1]
      \item there exists $k_0$ such that $D_{k_0}$ divides all $f_i$;
      \item $f_0$ is square-free;
      \item $f$ is irreducible and for each $k$, either $j_k$ or $v_k$ is an odd number;
    \end{enumerate}
    then $P\bigl(\R[x,y,z]/(f)\bigr)=\infty$.
\end{prop}
\begin{proof}
    Let $\tilde{f}=\sum_{i=0}^{n} f_{i}w^i$, and let $X=\Proj\bigl(\R[x,y,z,w]/(\tilde{f})\bigr)$.
    
    Let $X_k:=\Proj\bigl(\R[x,y,z]/(D_k)\bigr)$ be the irreducible components of $X_{w=0}$.
    Define $W_k=(X_k)_{\sm}\setminus \bigcup_{j\neq k} X_j$.
    
    (1) Let $k_0 \in [1,m]\cap \Z$ such that $D_{k_0}$ divides all $f_i$.
    
    Since $D_{k_0}$ changes sign in $\R^3$, $\R[x,y,z]/(D_{k_0})$ and $\R[x,y,z,w]/(D_{k_0})$ are formally real by Lemma~\ref{lSaR}.
    By Corollary~\ref{cRP}, $T=\Proj\bigl(\R[x,y,z,w]/(D_{k_0})\bigr)$ is a formally real conical surface.
    It is not difficult to see $\partial_{T(\C)}\bigl(T_{w\neq 0}(\R)\bigr)=T_{w=0}(\R)$.
    $T_{w=0}(\R)$ is Zariski dense in $T_{w=0}$ because $T_{w=0}$ is also formally real.
    Thus, $P\bigl(\R[x,y,z]/(D_{k_0})\bigr)=\infty$.
    Since $\R[x,y,z]/(D_{k_0})$ is a quotient of $\R[x,y,z]/(f)$, $P\bigl(\R[x,y,z]/(f)\bigr)=\infty$.
    
    (2) We may assume $f$ is irreducible because every irreducible factor of $f$ has the same property as $f$.
    Let $h \in [1,m]\cap \Z$.
    Since $X_k$ is formally real, we can choose a point $p \in W_h(\R)$.
    Fix coordinates of $p$ in $\R^3$ (or equivalently choose a suitable affine open subscheme of $\bP^2_{\R}$).
    
    Let $U \subseteq \R^3$ be a connected Euclidean open subset containing $p$ with compact closure.
    By the proof of Lemma~\ref{lSaR}, $f_0$ changes sign in $U$.
    Then there exists a positive real number $\delta$ such that $(-\delta,\delta) \subseteq f_0(U)$.
    Define \[\phi (w) = \max_{ q \in \overline{U}}\left( \Bigl| \sum_{i=i}^{n} f_{i}(q)w^i \Bigr| \right).\]
    $\phi$ is a continuous function on $\R$.
    Then, there is a positive real number $\epsilon$, such that $\phi\bigl((-\epsilon,\epsilon)\bigr)<\delta$.
    
    By the mean value theorem, for each $w \in (-\epsilon,\epsilon)$, there exists $q_w \in U$ with $\tilde{f}(q_w,w)=0$.
    This implies $p \in \partial_{X(\C)}\bigl(X_{w\neq 0}(\R)\bigr)$.
    Therefore, $\partial_{X(\C)}\bigl(X_{w\neq 0}(\R)\bigr)$ is Zariski dense in $X_{w=0}$.
    
    (3) Since $f$ is irreducible, no $D_k$ divides all $f_i$.
    Then for each $k$, $D_k$ does not divide $f_{v_k}$. 
    Define $H_k=\Proj\bigl(\R[x,y,z]/(f_{v_k})\bigr)$.
    
    Let $h \in [1,m]\cap \Z$. 
    Since $X_h$ is formally real, we can choose a real point $p \in W_h \setminus H_h$.
    Fix coordinates of $p$ in $\R^3$.
    Note that $f_{v_h}(p) \neq 0$.
    
    If $j_h$ is odd, then the proof is analogous to (2).
    We may assume $v_h$ is odd.
    
    Let $U \subseteq \R^3$ be a Euclidean open subset containing $p$, and let $\epsilon>0$ be a small positive real number.
    Define \[\phi_0 (x,y,z) = \min_{ w \in [-\epsilon,\epsilon]}\left(\sum_{i=v_h}^{n} f_{i}(x,y,z)w^i \right),\] and similar for $\phi_1$ with $\max$.
    $\phi_0$ and $\phi_1$ are continuous functions on $\R^3$.
    Since $v_h$ is an odd number, $\phi_0(p) < 0$ and $\phi_1(p)>0$.
    Thus, we have an open subset $V_0$ containing $p$ and a positive real number $\delta$ such that $\phi_0(V_0) < -\delta$ and $\phi_1(V_0)> \delta$.
    
    Let $f^{\circ}=\sum_{i<v_h} f_{i}w^i$.
    By the definition of $v_h$, $f^{\circ}(p)  = 0$ for all $w \in \R$.
    Thus, there exists an open subset $V_1\subseteq \R^3$ containing $p$ such that $\bigl|f^{\circ}(V_1,w) \bigr|< \delta$ for $w \in [-\epsilon,\epsilon]$.
    
    Let $V \subseteq U \cap V_0 \cap V_1$ be an open subset containing $p$.
    By the mean value theorem, for each $q \in V$, there exists $w_q \in [-\epsilon,\epsilon]$ with $\tilde{f}(q,w_q)=0$.
    This implies $p \in \partial_{X(\C)}\bigl(X_{w\neq 0}(\R)\bigr)$.
    Therefore, $\partial_{X(\C)}\bigl(X_{w\neq 0}(\R)\bigr)$ is Zariski dense in $X_{w=0}$.
\end{proof}

\section{Conjectures}

Let $k$ be a formally real field satisfying $P\bigl(k(t)\bigr) < \infty$. Let $A$ be a finitely generated $k$-algebra, and let $W$ denote the Zariski closure of formally real points in $\Spec(A)$. 
We have established that $P(A) < \infty$ if $\dim(W)=0$. For algebras with $\dim(A) = 1$, the finiteness of $P(A)$ remains conjectural (see \cite[Problem 2]{Choi82}).

It is known that $P(A)$ is finite if $\dim (W)\leq 1$ and $k$ is real closed by Corollary \ref{cDl1}.
If $\dim(W) \geq 3$, then $P(A)= \infty$ by \cite[Theorem 6.6]{Choi82}.
Thus, the case $\dim(W) = 2$ remains to be studied (see \cite[Problem 4]{Choi82}).
Building on \cite[Conjecture 3.3]{Blac24} which reframes this problem, and supported by evidence from \cite[Theorem 4.16]{Choi82}, \cite[Theorem 2.3]{Blac24} and Theorem~\ref{MT}, we conjecture:
\begin{conj}\label{conj1}
	$P(A) < \infty \iff\dim(W)<2$.
\end{conj}

This provides a complete dimensional framework for predicting the finiteness of Pythagoras numbers based on the Krull dimension of formally real points.
When $k=\R$,  Conjecture \ref{conj1} admits five equivalent formulations.

\begin{prop} \label{pCt1}
The following conjectures are equivalent:
    \begin{enumerate}[1]
    \item Conjecture~\ref{conj1} holds over $\R$.
    \item For any finitely generated two-dimensional formally real domain $A$ over $\R$, $P(A) = \infty$.
    \item Let $f\in \R[X,Y,Z]$ be irreducible and $g \in \R[X,Y,Z]$ satisfy $f \nmid g$. 
        If $B=\R[X,Y,Z]_g/(f)$ is formally real, then $P(B) = \infty$.
    \item With notation as in (3), if $f$ changes sign in $\R^3$,  then $P(B) = \infty$ (see a similar version in \cite[Conjecture 3.3]{Blac24}).
    \item Let $A$ be a finitely generated two-dimensional formally real regular domain over $\R$ with $P(A)=\infty$. Then for every nonzero $ f \in A$,  $P(A_f) = \infty$.
    \end{enumerate}
\end{prop}
\begin{proof}
    $(1) \Rightarrow (2)$, $(2) \Rightarrow (3)$ and $(2) \Rightarrow (5)$ are straightforward.
    
    \medskip \noindent $(2) \Rightarrow (1)$. Let $C$ be a finitely generated algebra over $\R$, and let $W$ be the Zariski closure of the real points of $\Spec(C)$.
    If $\dim(W) < 2$, then $P(C)<\infty$ by Corollary~\ref{cDl1}.
    
    If $\dim(W) \geq 2$, then $C$ has a formally real integral quotient $A$ by Proposition~\ref{prop1_2}(1, 4), whose Krull dimension exceeds $1$.
    If $\dim(A) > 2$, then $P(A) = \infty$ (see \cite[Theorem 6.6]{Choi82}).
    If $\dim(A)= 2$, then $P(A) = \infty$ by (2). Thus, $P(C)=\infty$.
    
    \medskip \noindent $(3) \Rightarrow (2)$. 
    Let $A$ be a finitely generated two-dimensional integral domain over $\R$, and let $W$ be the Zariski closure of the real points of $\Spec(A)$.
    The fraction field of $A$ is a finite extension of $\R(X,Y)$.
    Therefore, there is a localization $D$ of $A$ such that $D=\R[X,Y,Z]_{g'}/(f')$, where $f',g' \in \R[X,Y,Z]$.
    
    If $A$ is formally real, then $D$ is formally real since they share the same function field.
    Hence, $P(A) \geq P(D) = \infty$ by (3).
    
    \medskip \noindent $(3) \Leftrightarrow (4)$.  We need to prove that $B$ is formally real $\Leftrightarrow$ $f$ changes sign in $\R^3$, which is established in Lemma~ \ref{lSaR}.
    
    \medskip \noindent $(5) \Rightarrow (2)$. 
    Let $A$ be a finitely generated two-dimensional domain over $\R$.
    Then $\Spec(A)$ is birational to a projective smooth surface $X \hookrightarrow \bP^n_{\R}$.
    Let $s \in \Gamma(X, \bO_X(1))$ such that $X_{s=0}$ is a formally real integral curve.
    Let $B=\Gamma(X_{s \neq 0})$.
    Then $P(B)=\infty$ by Corollary~\ref{cMSC}.
    Let $U$ be a common principal open subscheme of $X_{s\neq 0}$ and $\Spec(A)$.
    Let $C=\Gamma(U)$.
    Then $P(C) \leq P(A)$.
    Since $P(C)=\infty$ by (5), it follows that $P(A)= \infty$.
\end{proof}
\section{Statements and Declarations}
\subsection{Ethical Statement}
The work adheres to the journal's ethical standards and principles of academic integrity.
\subsection{Funding and Conflicts of Interest}
This work was partially supported by the National Natural Science Foundation of China (Grant No. 12371013) and the Innovation Program for Quantum Science and Technology (Grant No. 2021ZD0302902).
The authors declare no financial or non-financial conflicts of interest that could influence the content or conclusions of this work.
\subsection{Data Accessibility}
No datasets were generated or analyzed during this study. Data sharing is not applicable to this purely theoretical work.


\begin{thebibliography}{99}

\bibitem{Beno20}
O. Benoist. {\em Sums of squares in function fields over Henselian local fields}. Math. Annal.
\textbf{376}(2020), 683-692.
	
\bibitem{Blac24}
K. B\l{}achut, T. Kowalczyk. {\em Sums of squares on hypersurfaces}.
Results in Math. \textbf{79}(2024), no. 2, Paper No. 90, 11 pp.

\bibitem{Blek16}
G. Blekherman, G. G. Smith, M. Velasco. {\em Sums of squares and varieties of minimal degree}.
J. Amer. Math. Soc. \textbf{29}(2016), 893-913.
	
\bibitem{Boch98}
J. Bochnak, M. Coste, M.-F. Roy. {\em Real Algebraic Geometry}. Ergeb. Math. Grenzgeb. (3), \textbf{36}.
Springer-Verlag, Berlin, 1998.

\bibitem{Choi82}
M. D. Choi, Z. D. Dai, T. Y. Lam, B. Reznick. {\em  The Pythagoras number of some affine algebras and local algebras}. J. Reine Angew. Math. \textbf{336}(1982), 45-82.

\bibitem{Choi95}
M. D. Choi, T. Y. Lam, B. Reznick. {\em Sums of squares of real polynomials}. Proc. Sympos. Pure Math. \textbf{58.2}(1995), 103-126.

\bibitem{Corn86}
G. Cornell, J. H. Silverman {\em Arithmetic Geometry}. Springer-Verlag, New York, 1986.
	
\bibitem{Enge78}
R. Engelking. {\em Dimension Theory}. Mathematical Studies \textbf{19}.
North-Holland Publishing Company, 1978.

\bibitem{Fant05}
B. Fantechi, L. G\`{o}ttsche, L. Illusie, S. L. Kleiman, N.Nitsure, A. Vistoli {\em Fundamental Algebraic Geometry: Grothendieck's FGA Explained}. Mathematical surveys and monographs \textbf{123}. American Mathematical Soc., 2005.

\bibitem{Grau84}
H. Grauert, R. Remmert {\em Coherent Analytic Sheaves}. Grundlehren der mathematischen Wissenschaften \textbf{265}. Springer-Verlag, 1984.

\bibitem{Hart77}
R. Hartshorne. {\em Algebraic Geometry}. Graduate Texts in
Mathematics \textbf{52}. Springer-Verlag, New York, 1977.

\bibitem{Hu15}
Y. Hu. {\em The Pythagoras number and the u-invariant of Laurent series fields in several varibles}. J. Algebra.
\textbf{426}(2015), 243-258.

\bibitem{Lair21}
P. Lairez, M. Safey El Din. {\em Computing the dimension of real algebraic sets}. ISSAC 2021 -46th International Symposium on Symbolic and Algebraic Computation, Jul 2021, Saint-Pétersbourg,Russia. pp.257-264.

\bibitem{Lam05}
T. Y. Lam. {\em Introduction to Quadratic Forms over Fields}. Graduate Studies in Mathematics \textbf{67}. American Mathematical Soc., 2005.

\bibitem{Munk14}
J. R. Munkres.{\em Topology}. Pearson Education Limited, 2nd edition, 2014.

\bibitem{Sche01}
C. Scheiderer. {\em On sums of squares in local rings}.
J. Reine Angew. Math. \textbf{540}(2001), 205-227.
	
\bibitem{Sche06}
C. Scheiderer. {\em Sums of squares on real algebraic surfaces}.
Manuscripta Math. \textbf{119}(2006), no.4, 395-410.

\bibitem{Serr56}
J-P. Serre. {\em G\'{e}om\'{e}trie alg\'{e}brique et G\'{e}om\'{e}trie analytique}. Annales de l'Institut Fourier \textbf{6}(1956), 1-42.

\bibitem{stacks-project}
The {Stacks Project Authors}. \textit{Stacks Project}.
\url{https://stacks.math.columbia.edu}, 2018.
\end{thebibliography}
\end{document}